\documentclass[preprint,11pt]{elsarticle}

\usepackage{amsmath}
\usepackage{amsfonts}
\usepackage{amssymb}
\usepackage{amsthm}
\usepackage{booktabs}
\usepackage{ragged2e}
\usepackage[utf8]{inputenc}
\usepackage[english]{babel}
\usepackage{graphicx}
 \usepackage{parskip}
\usepackage{textpos}
\usepackage[utf8]{inputenc}
\usepackage[T1]{fontenc}
\usepackage{helvet}
\usepackage{amsfonts}
\usepackage{amsmath}
\usepackage{amssymb}
\usepackage{tikz, graphicx}
\usepackage{hyperref}
\biboptions{numbers,sort&compress}
\usepackage{bigints}
\usepackage{xcolor}
\usepackage{algorithm}
\usepackage{algpseudocode}
\usepackage{geometry}
\newtheorem{theorem}{Theorem}[section]

\newtheorem{definition}[theorem]{Definition}
\newtheorem{corollary}[theorem]{Corollary}

\newtheorem{lemma}[theorem]{Lemma}

\newtheorem{remark}[theorem]{Remark}

\usepackage{float}
\usepackage{graphicx} % Required for inserting images
\numberwithin{equation}{section}

\begin{document}
\begin{frontmatter}

%% Title, authors and addresses

%% use the tnoteref command within \title for footnotes;
%% use the tnotetext command for theassociated footnote;
%% use the fnref command within \author or \affiliation for footnotes;
%% use the fntext command for theassociated footnote;
%% use the corref command within \author for corresponding author footnotes;
%% use the cortext command for theassociated footnote;
%% use the ead command for the email address,
%% and the form \ead[url] for the home page:
%% \title{Title\tnoteref{label1}}
%% \tnotetext[label1]{}
%% \author{Name\corref{cor1}\fnref{label2}}
%% \ead{email address}
%% \ead[url]{home page}
%% \fntext[label2]{}
%% \cortext[cor1]{}
%% \affiliation{organization={},
%%             addressline={},
%%             city={},
%%             postcode={},
%%             state={},
%%             country={}}
%% \fntext[label3]{}

\title{Stancu-Type Generalizations of Neural Network Operators with Perturbed Sampling Nodes} %% Article title

%Multivariate

%% Authors
\author{Sachin Saini}
\ead{sachin\_s@ma.iitr.ac.in; sbisaini003@gmail.com}

% \affiliation{organization={Department of Mathematics, Indian Institute of Technology Roorkee},
%             city={Roorkee},
%             postcode={247667},
%             state={Uttarakhand},
%             country={India}}

\begin{abstract}
In this paper, we introduce a Stancu-type generalization of multivariate neural network operators (NNOs) by incorporating two parameters that perturb the sampling nodes. The proposed operators extend the existing NNOs by allowing flexibility in the placement of sampling nodes. We establish the well-definedness and boundedness of the operators and prove uniform convergence on compact domains. Furthermore, quantitative error estimates are derived in terms of the modulus of continuity, leading to a convergence rate result.  A perturbation theorem is also established, showing how the Stancu-type operators differ from the existing NNOs. Numerical experiments are presented to illustrate the influence of the Stancu parameters on the approximations and to validate the theoretical results. Finally, an application to ECG signal reconstruction is considered, demonstrating that the proposed operators effectively preserve the essential characteristics of the signal while providing accurate approximations.
\end{abstract}


\end{frontmatter}

\section{Introduction}
Neural networks have become a fundamental tool in modern applied mathematics, scientific computing, and data-driven modelling due to their remarkable approximation capabilities. Classical results such as the universal approximation theorem demonstrate that feedforward neural networks with a single hidden layer are capable of approximating a wide class of continuous functions with arbitrary accuracy. Beyond the traditional training-based perspective, neural networks can also be studied within a rigorous mathematical framework through the theory of \emph{neural network operators} (NNOs). In this setting, neural networks are interpreted as explicit approximation operators acting on function spaces, allowing the approximation process to be analyzed by means of classical tools from approximation theory. This operator-based approach provides constructive approximation schemes together with precise convergence results and quantitative error estimates.

The systematic study of NNOs was initiated by Cardaliaguet and Euvrard~\cite{cardaliaguet1992approximation}, who introduced bell-shaped and squashing-type neural constructions and established an important connection between neural networks and classical approximation theory. Subsequently, Anastassiou developed a comprehensive approximation theory for NNOs
\cite{anastassiou1997rate, anastassiou2000rate,
anastassiou2011multivariatesigmoidal,
anastassiou2013frac-normal, anastassiou2013multi-vari-rate}.
His contributions include the derivation of convergence rates for various sigmoidal activation functions and extensions of the framework to multivariate, fractional, and hyperbolic NNOs. In recent years, the theory has witnessed significant advancements, including density results, quantitative estimates, and nonlinear extensions of NNOs, which have substantially broadened the scope of the approximation framework~\cite{costarelli2022density,costarelli2022quantitative,costarelli2024Durrmeyer,berisha2025riemann,saini2026approximation,yu2025construction,costarelli2023approximation,bhat2025boundary}. These developments have enabled the study of a wider class of activation functions and facilitated more general operator-based constructions.

A particularly significant contribution in this direction was made by Costarelli and Spigler~\cite{costarelli2013approxi.Single-layer-for-C-class}, who introduced normalized positive sampling-type NNOs acting on the space $C(\mathcal{I})$, where
$\mathcal{I}=[a,b]\subset\mathbb{R}$. These operators are defined by
\begin{equation}\label{classicalLNNO1}
F_n(f;s) = \frac{\sum\limits_{k=\lceil na\rceil}^{\lfloor nb\rfloor}
\sigma_\eta(ns-k)\,f\!\left(\tfrac{k}{n}\right)} {\sum\limits_{k=\lceil na\rceil}^{\lfloor nb\rfloor} \sigma_\eta(ns-k)}, \quad s\in\mathcal{I}.
\end{equation}
Under suitable assumptions on the sigmoidal activation function, these operators preserve positivity and constants and approximate every function $f\in C(\mathcal{I})$ uniformly on $\mathcal{I}$. 

Existing NNOs are generally constructed using uniform sampling nodes of the form $k/n$. However, such a formulation relies on the assumption that the sampling locations are known exactly and follow an ideal uniform grid. In many approximation and signal reconstruction problems, the sampling nodes may be subject to perturbations or deviations from this ideal configuration. Therefore, it is natural to investigate the stability and approximation behavior of NNOs when the underlying sampling nodes are perturbed.

Motivated by this consideration, we introduce a perturbation framework for NNOs by replacing the usual sampling nodes $k/n$ with the Stancu-type perturbed nodes
\[ s_{k,n}^{\alpha,\beta}=\frac{k+\alpha}{n+\beta}. \]
These parameters provide a controlled mechanism to modify the locations of the sampling points while preserving the positive sampling structure of the original NNOs. This formulation enables us to analyze the influence of node perturbations on the approximation properties and stability of NNOs.

The idea of introducing such perturbations is inspired by the classical Stancu generalization of positive linear operators, where additional parameters are employed to modify the sampling nodes and improve the flexibility of approximation schemes. Several Stancu-type extensions of approximation operators have been investigated in recent years \cite{nasiruzzaman2025approximation,torun2024some,lian2023approximation,cai2024approximation}.
Inspired by these developments, the aim of this paper is to introduce and study a multivariate Stancu-type generalization of NNOs defined in \eqref{classicalLNNO1}. The proposed operators incorporate two Stancu parameters that allow a controlled perturbation of the sampling nodes, and we establish their fundamental approximation properties, including well-definedness, boundedness, convergence, and quantitative error estimates.

\section{Preliminaries}\label{sec:preliminaries}
Let $\mathcal I=[a,b]\subset\mathbb R$ be a compact interval and  \[ K=\prod_{i=1}^{d}[a_i,b_i]\subset\mathbb R^d \] be a compact rectangular domain. $C(K)$ be the Banach space of all real-valued continuous functions on $K$, equipped with the uniform norm
\[ \|f\|_{\infty} := \sup_{\mathbf s\in K}|f(\mathbf s)|. \]
A measurable function $\eta:\mathbb R\to\mathbb R$ is called
\emph{sigmoidal} if
\[ \lim_{s\to-\infty}\eta(s)=0, \quad \lim_{s\to+\infty}\eta(s)=1. \]
Throughout this paper, we assume that $\eta$ is an increasing sigmoidal function satisfying $\eta(1)<1$ and the following conditions:
\begin{enumerate}
\item[(P1)]
$\eta(s)-\frac12$ is an odd function;
\item[(P2)]
$\eta\in C^2(\mathbb R)$ and $\eta$ is concave on $[0,\infty)$;
\item[(P3)]
there exists $\rho>0$ such that
\[ \eta(s) = \mathcal O\left(|s|^{-1-\rho}\right), \quad s\to-\infty. \]
\end{enumerate}
By condition \emph{(P1)}, we also have
\[ 1-\eta(s) = \mathcal O\left(|s|^{-1-\rho}\right), \quad s\to+\infty.\]
Associated with $\eta$, we define the neural activation function
\[\sigma_\eta(s) := \frac12\left[\eta(s+1)-\eta(s-1)\right], \quad s\in\mathbb R.\]
For $d\geq1$, the corresponding multivariate activation function is defined by
\[ \sigma(\mathbf s) = \prod_{i=1}^{d}\sigma_\eta(s_i), \quad
\mathbf s=(s_1,\ldots,s_d)\in\mathbb R^d. \]
\begin{definition}
For $r\geq0$, the \emph{discrete absolute moment of order $r$}
associated with $\sigma_\eta$ is defined by
\[ \mathcal M_r(\sigma_\eta) := \sup_{s\in\mathbb R} \sum_{k\in\mathbb Z}
|\sigma_\eta(s-k)|\,|s-k|^r. \]
\end{definition}
Under assumptions \emph{(P1)-(P3)}, it follows that
\[ \mathcal M_r(\sigma_\eta)<\infty, \quad 0\leq r<\rho. \]
For further details, we refer to
\cite{costarelli2013approxi.Single-layer-for-C-class,saini2026approximation}.

Since $\eta(s)-\frac12$ is odd, differentiation gives that $\eta'$ is even. Moreover, by \emph{(P2)}, $\eta'$ is non-increasing on $[0,\infty)$.

Differentiating $\sigma_\eta$, we obtain
\[ \sigma_\eta'(s) = \frac12\left(\eta'(s+1)-\eta'(s-1)\right). \]
For $s\geq0$, since \[ |s+1|\geq |s-1|, \] and $\eta'$ is even and non-increasing, we have
\[ \eta'(s+1)\leq\eta'(s-1). \]
Consequently,
\[ \sigma_\eta'(s)\leq0, \quad s\geq0, \] which gives that under assumptions \emph{(P1)-(P2)}, $\sigma_\eta$ is non-increasing on $[0,\infty)$.

\begin{lemma}\label{lemma-psi-properties}
The activation function $\sigma_\eta$ and $\sigma$ satisfies the following properties:

\begin{enumerate}
\item[(i)] $\sigma_\eta(s)\geq0$ for every $s\in\mathbb R$, and \[ \lim_{s\to\pm\infty}\sigma_\eta(s)=0; \]
\item[(ii)] $\sigma_\eta$ is an even function;

\item[(iii)] for every $s\in\mathbb R$,
\[ \sum_{k\in\mathbb Z}\sigma_\eta(s-k)=1, \] and \( \|\sigma_\eta\|_{L^1(\mathbb R)}=1; \)

\item[(iv)] for every $\delta>0$, \[ \lim_{n\to\infty} \sup_{s\in\mathbb R}
\sum_{\substack{k\in\mathbb Z\\ |ns-k|>n\delta}} \sigma_\eta(ns-k) = 0;
\] Moreover \[ \lim_{n\to\infty} \sup_{\mathbf s\in\mathbb R^d} \sum_{\substack{\mathbf k\in\mathbb Z^d\\
\|\mathbf{s}n-\mathbf{k}\|_\infty>n\delta}}
\sigma(n\mathbf{s}-\mathbf{k}) = 0 .\]

\item[(v)] for every $s\in[a,b]$ and $n\in\mathbb N$,
\[ 1\geq \sum_{k=\lceil na\rceil}^{\lfloor nb\rfloor} \sigma_\eta(ns-k)
\geq \sigma_\eta(1)>0; \]

\item[(vi)] \( \sigma_\eta(s) = \mathcal O\left(|s|^{-1-\rho}\right),
\quad s\to\pm\infty; \)

\item[(vii)] the series \( \sum\limits_{k\in\mathbb Z}\sigma_\eta(s-k) \) converges uniformly on compact subsets of $\mathbb R$;

\item[(viii)] $\sigma_\eta$ is Lipschitz continuous on $\mathbb R$.
\end{enumerate}
\end{lemma}

The above properties are standard for sigmoidal activation functions
satisfying \emph{(P1)-(P3)}; see \cite{costarelli2013approxi.Single-layer-for-C-class, costarelli2014convergence,costarelli2013multivariate}.

% \begin{lemma}\label{lem:multi-localization}
% Let \( \sigma(\mathbf x) = \prod\limits_{i=1}^{d}\sigma_\eta(x_i), \ \mathbf x=(x_1,\ldots,x_d)\in\mathbb R^d. \)
% Then, for every $\delta>0$,
% \[ \lim_{n\to\infty} \sup_{\mathbf s\in\mathbb R^d} \sum_{\substack{\mathbf k\in\mathbb Z^d\\
% \|\mathbf{s}n-\mathbf{k}\|_\infty>n\delta}}
% \sigma(n\mathbf{s}-\mathbf{k}) = 0 .\]
% \end{lemma}

% \begin{proof}
% The result follows from property \emph{(iv)} of Lemma~\ref{lemma-psi-properties} together with the product structure of the multivariate activation function.
% \end{proof}

\section{Multivariate Stancu-Type NNOs}\label{sec:stancu_operator}

In this section, we introduce a multivariate Stancu-type modification of NNOs. The Stancu parameters perturb the sampling nodes and provide a mechanism for studying the influence of non-uniform sampling on the approximation process.

Let
\[ \boldsymbol{\alpha} = (\alpha_1,\ldots,\alpha_d), \quad \boldsymbol{\beta}
= (\beta_1,\ldots,\beta_d) \] satisfy \( 0\leq\alpha_i\leq\beta_i, \quad i=1,\ldots,d. \)

For each \(n\in\mathbb N\), define the index set
\[ \Lambda_n = \left\{ \mathbf k=(k_1,\ldots,k_d)\in\mathbb Z^d: \lceil na_i\rceil\leq k_i\leq\lfloor nb_i\rfloor, \ i=1,\ldots,d \right\}.\]
For \(\mathbf k\in\Lambda_n\), the Stancu-type perturbed sampling nodes are
defined by
\[ \mathbf s_{\mathbf k,n}^{(\boldsymbol{\alpha},\boldsymbol{\beta})} =
\left( \frac{k_1+\alpha_1}{n+\beta_1}, \ldots, \frac{k_d+\alpha_d}{n+\beta_d}
\right).\]
Define the compact neighbourhood of $K$
\[ K^* = \prod_{i=1}^{d} [a_i-\beta_i,b_i+\beta_i]. \]
Clearly, \[ K\subset K^*. \]
Moreover, since
\[ \frac{k_i+\alpha_i}{n+\beta_i} \in [a_i-\beta_i,b_i+\beta_i], \]
we have
\[ \mathbf s_{\mathbf k,n}^{(\boldsymbol{\alpha},\boldsymbol{\beta})} \in K^* \]
for every \(n\in\mathbb N\) and \(\mathbf k\in\Lambda_n\).

The following estimate describes the deviation of the perturbed nodes from the classical sampling grid.

\begin{lemma}\label{lem:stancu-node}
For every \(n\in\mathbb N\) and every \(\mathbf k\in\Lambda_n\), there exists a constant
\(\widetilde C>0\), independent of \(n\) and \(\mathbf k\), such that
\[ \left\|
\mathbf s_{\mathbf k,n}^{(\boldsymbol{\alpha},\boldsymbol{\beta})} -
\frac{\mathbf k}{n} \right\|_\infty
\le \frac{\widetilde C}{n}.
\]
\end{lemma}

\begin{proof}
For each coordinate \(i=1,\ldots,d\),
\[ \frac{k_i+\alpha_i}{n+\beta_i} - \frac{k_i}{n} = \frac{n\alpha_i-k_i\beta_i}{n(n+\beta_i)}. \]
Let \( M_i=\max\{|a_i|,|b_i|\}. \)
Since \(\mathbf k\in\Lambda_n\), we have \(|k_i|\le nM_i\). Therefore,
\[
\begin{aligned}
\left| \frac{k_i+\alpha_i}{n+\beta_i} - \frac{k_i}{n} \right|
&=\frac{|\,n\alpha_i-k_i\beta_i\,|} {n(n+\beta_i)}  \\
&\le \frac{n\alpha_i+\beta_i|k_i|}{n(n+\beta_i)}  \\
&\le \frac{\alpha_i+\beta_iM_i} {n+\beta_i}
\le \frac{C_i}{n},
\end{aligned}
\]
where \( C_i:=\alpha_i+\beta_iM_i, \) which is independent of \(n\) and \(k_i\). Finally,  we obtain
\[ \left\|
\mathbf s_{\mathbf k,n}^{(\boldsymbol{\alpha},\boldsymbol{\beta})} -
\frac{\mathbf k}{n} \right\|_\infty = \max_{1\le i\le d}
\left| \frac{k_i+\alpha_i}{n+\beta_i} - \frac{k_i}{n} \right|
\le \frac{\widetilde C}{n},
\]
with, \( \widetilde C:=\max\limits_{1\le i\le d} C_i. \) 
\end{proof}

\begin{definition}
Let \( f\in C(K^*). \) The multivariate Stancu-type NNOs is defined by
\[
F_n^{(\boldsymbol{\alpha},\boldsymbol{\beta})}
(f;\mathbf s) = \frac{ \displaystyle \sum_{\mathbf k\in\Lambda_n}
\sigma(n\mathbf s-\mathbf k) f\left( \mathbf s_{\mathbf k,n}^{(\boldsymbol{\alpha},\boldsymbol{\beta})} \right) } {\displaystyle \sum_{\mathbf k\in\Lambda_n} \sigma(n\mathbf s-\mathbf k) }, \quad
\mathbf s\in K.
\]
\end{definition}

\begin{remark}
For \( \boldsymbol{\alpha}=\boldsymbol{\beta}=\mathbf0, \) the perturbed nodes reduce to the classical sampling points:
\[ \mathbf s_{\mathbf k,n}^{(\mathbf0,\mathbf0)} = \frac{\mathbf k}{n}. \]
Consequently,
\[ F_n^{(\mathbf0,\mathbf0)} = F_n, \] and the proposed operators form a Stancu-type extension of the existing multivariate NNOs.
\end{remark}
The Stancu parameters modify only the sampling nodes and do not affect the
normalizing denominator. Hence we define
\[ D_n(\mathbf s) = \sum_{\mathbf k\in\Lambda_n} \sigma(n\mathbf s-\mathbf k),
\quad \mathbf s\in K.\]
The following lemma establishes the uniform boundedness of the denominator.

\begin{lemma}\label{lem:denominator}
For every \(n\in\mathbb N\) and every \(\mathbf s\in K,\) there exist constants \(m_\sigma,M_\sigma>0\), independent of
\(n\) and \(\mathbf s\), such that
\[ 0<m_\sigma \leq D_n(\mathbf s) \leq M_\sigma<\infty.\]
\end{lemma}

\begin{proof}
Using the product structure of the multivariate activation function, we have
\[ \sigma(n\mathbf s-\mathbf k) = \prod_{i=1}^{d} \sigma_\eta(ns_i-k_i).\]
Since \( s_i\in[a_i,b_i], \) property \emph{(v)} of Lemma~\ref{lemma-psi-properties} yields
\[ \sum_{k_i=\lceil na_i\rceil}^{\lfloor nb_i\rfloor} \sigma_\eta(ns_i-k_i)
\geq \sigma_\eta(1),
\]
for every \(i=1,\ldots,d\).

Consequently,
\[ D_n(\mathbf s) \geq \prod_{i=1}^{d}\sigma_\eta(1) =
\bigl(\sigma_\eta(1)\bigr)^d. \]
Therefore, by setting \( m_\sigma = \bigl(\sigma_\eta(1)\bigr)^d, \) we obtain
\[ D_n(\mathbf s)\geq m_\sigma>0. \]
For the upper estimate, we have
\[ \sum_{k_i=\lceil na_i\rceil}^{\lfloor nb_i\rfloor} \sigma_\eta(ns_i-k_i)
\leq \sum_{k_i\in\mathbb Z} \sigma_\eta(ns_i-k_i). \]
Since \(\sigma_\eta\geq0\), by the definition of the discrete absolute
moment,
\[ D_n(\mathbf s) \leq \prod_{i=1}^{d} \mathcal M_0(\sigma_\eta) =
\bigl(\mathcal M_0(\sigma_\eta)\bigr)^d.\]
Taking \( M_\sigma = \bigl(\mathcal M_0(\sigma_\eta)\bigr)^d, \) we conclude that
\[ 0<m_\sigma \leq D_n(\mathbf s)
\leq M_\sigma<\infty, \] uniformly for all \( n\in\mathbb N, \
\mathbf s\in K. \)
\end{proof}

The previous lemma guarantees that the denominator of the proposed operator is uniformly positive. Therefore, the multivariate Stancu-type NNOs are well defined on \(C(K^*)\).

\begin{lemma}\label{lemma:stancu-well-defined}
Let \( f\in C(K^*). \) Then, for every \(n\in\mathbb N\) and every
\(\mathbf s\in K\), the operator
\[ F_n^{(\boldsymbol\alpha,\boldsymbol\beta)} (f;\mathbf s) \] is bounded. Moreover,
\[ \left| F_n^{(\boldsymbol\alpha,\boldsymbol\beta)} (f;\mathbf s)
\right| \leq \left( \frac{\mathcal M_0(\sigma_\eta)}
{\sigma_\eta(1)} \right)^d \|f\|_\infty .\]
\end{lemma}

\begin{proof}
Using the definition of the operator and the triangle inequality, we obtain
\begin{align}\nonumber
\left|
F_n^{(\boldsymbol\alpha,\boldsymbol\beta)} (f;\mathbf s) \right|
&= \left| \frac{ \displaystyle \sum_{\mathbf k\in\Lambda_n}
\sigma(n\mathbf s-\mathbf k) f\left( \mathbf s_{\mathbf k,n}^{(\boldsymbol\alpha,\boldsymbol\beta)} \right) }
{D_n(\mathbf s)} \right| \\ \label{eq: boundedness_1}
&\leq \frac{ \displaystyle \sum_{\mathbf k\in\Lambda_n}
\sigma(n\mathbf s-\mathbf k) \left| f\left(
\mathbf s_{\mathbf k,n}^{(\boldsymbol\alpha,\boldsymbol\beta)} \right)
\right|}{ D_n(\mathbf s)}.
\end{align}

Because \(
\mathbf s_{\mathbf k,n}^{(\boldsymbol\alpha,\boldsymbol\beta)}
\in K^*, \) we have
\[ \left| f\left( \mathbf s_{\mathbf k,n}^{(\boldsymbol\alpha,\boldsymbol\beta)}
\right) \right| \leq
\|f\|_\infty.
\]
Therefore, equation~\eqref{eq: boundedness_1} becomes
\[ \left| F_n^{(\boldsymbol\alpha,\boldsymbol\beta)}
(f;\mathbf s) \right| \leq \frac{ \|f\|_\infty
\displaystyle \sum_{\mathbf k\in\Lambda_n} \sigma(n\mathbf s-\mathbf k)
} { D_n(\mathbf s)}.
\]
Using the definition of the zeroth discrete moment, we obtain
\[ \sum_{\mathbf k\in\Lambda_n} \sigma(n\mathbf s-\mathbf k)
\leq \bigl(\mathcal M_0(\sigma_\eta)\bigr)^d.\]
Together with \( D_n(\mathbf s) \geq \bigl(\sigma_\eta(1)\bigr)^d, \)
we conclude that
\[ \left| F_n^{(\boldsymbol\alpha,\boldsymbol\beta)}
(f;\mathbf s) \right| \leq \left( \frac{\mathcal M_0(\sigma_\eta)}
{\sigma_\eta(1)} \right)^d \|f\|_\infty.
\]
This completes the proof.
\end{proof}

The previous results allow us to investigate the approximation behaviour of
the proposed Stancu-type NNOs. We first establish their
uniform convergence on \(C(K^*)\) and then derive quantitative estimates in
terms of the modulus of continuity.

For \(f\in C(K^*)\), the modulus of continuity of \(f\) on \(K^*\) is defined
by
\[ \omega(f,\delta) := \sup_{\substack{\mathbf u,\mathbf v\in K^*\\
\|\mathbf u-\mathbf v\|_\infty\leq\delta}} |f(\mathbf u)-f(\mathbf v)|,
\quad \delta>0. \]
\begin{theorem}\label{thm:uniform-convergence} 
Let \( f\in C(K^*). \) Then
\[ \lim_{n\to\infty} F_n^{(\boldsymbol\alpha,\boldsymbol\beta)}
(f;\mathbf s) = f(\mathbf s) \]
uniformly for \( \mathbf s\in K. \)
\end{theorem}

\begin{proof}
Let \(\mathbf s\in K\). From the definition of the operator, and Lemma~\ref{lem:denominator}, we have

\begin{align}\nonumber
\left| F_n^{(\boldsymbol\alpha,\boldsymbol\beta)} (f;\mathbf s) -
f(\mathbf s) \right| 
& = \left| \frac{ \displaystyle \sum_{\mathbf k\in\Lambda_n}
\sigma(n\mathbf s-\mathbf k) \left(f(\mathbf s_{\mathbf k,n}^{(\boldsymbol\alpha,\boldsymbol\beta)}) - f(\mathbf s)
\right) } {D_n(\mathbf s) } \right|\\ \label{eq: I}
& \leq \frac{1}{(\sigma_\eta(1))^d} \sum_{\mathbf k\in\Lambda_n}
\sigma(n\mathbf s-\mathbf k)
\left| f( \mathbf s_{\mathbf k,n}^{(\boldsymbol\alpha,\boldsymbol\beta)})
- f(\mathbf s) \right|.
\end{align}

Since \(K^*\) is compact and \( \mathbf s_{\mathbf k,n}^{(\boldsymbol\alpha,\boldsymbol\beta)}
\in K^*, \) the function \(f\) is uniformly continuous on \(K^*\). Then
for any \(\varepsilon>0,\) there exist 
\(\delta>0\) such that
\[
|f(\mathbf u)-f(\mathbf v)|<\varepsilon , \text{ whenever } 
\|\mathbf u-\mathbf v\|_\infty<\delta
\]
for all \(\mathbf u,\mathbf v\in K^*\).

By Lemma~\ref{lem:stancu-node}, there exists \(\widetilde  C>0\) such that
\[ \left\|
\mathbf s_{\mathbf k,n}^{(\boldsymbol\alpha,\boldsymbol\beta)} -
\frac{\mathbf k}{n} \right\|_\infty \leq \frac{ \widetilde C}{n}.
\]
Therefore,
\[ \left\|
\mathbf s_{\mathbf k,n}^{(\boldsymbol\alpha,\boldsymbol\beta)} -
\mathbf s \right\|_\infty \leq \left\| \frac{\mathbf k}{n} -
\mathbf s \right\|_\infty +
\frac{ \widetilde C}{n}.
\]
Now, we divide the summation of \eqref{eq: I} into two parts as:
\begin{align*}
\sum_{\mathbf k\in\Lambda_n}\sigma(n\mathbf s-\mathbf k)
\left| f( \mathbf s_{\mathbf k,n}^{(\boldsymbol\alpha,\boldsymbol\beta)})
- f(\mathbf s) \right|& = \sum_{\mathbf k\in\Lambda_n^{(1)}}
\sigma(n\mathbf s-\mathbf k)
\left| f( \mathbf s_{\mathbf k,n}^{(\boldsymbol\alpha,\boldsymbol\beta)})
- f(\mathbf s) \right|\\ 
& \qquad\qquad\qquad  + \sum_{\mathbf k\in\Lambda_n^{(2)}}
\sigma(n\mathbf s-\mathbf k)
\left| f( \mathbf s_{\mathbf k,n}^{(\boldsymbol\alpha,\boldsymbol\beta)})
- f(\mathbf s) \right|\\
&=:I_1+I_2,
\end{align*}
where \( \Lambda_n^{(1)} = \left\{ \mathbf k\in\Lambda_n: \|n\mathbf s-\mathbf k\|_\infty \leq \frac{n\delta}{2} \right\}, \)
\(\Lambda_n^{(2)} = \left\{ \mathbf k\in\Lambda_n: \|n\mathbf s-\mathbf k\|_\infty > \frac{n\delta}{2} \right\}\) and \[ \Lambda_n = \Lambda_n^{(1)} \cup \Lambda_n^{(2)}.\]

For sufficiently large \(n\), \( \frac{\widetilde C}{n}<\frac{\delta}{2}, \) therefore, for
\(\mathbf k\in\Lambda_n^{(1)}\),
we have
\[ \left\| \mathbf s_{\mathbf k,n}^{(\boldsymbol\alpha,\boldsymbol\beta)} -
\mathbf s \right\|_\infty < \delta, \] and consequently,
\[ \left| f( \mathbf s_{\mathbf k,n}^{(\boldsymbol\alpha,\boldsymbol\beta)}) -
f(\mathbf s) \right| <
\varepsilon. \]
Hence,
\begin{align}\label{eq: I_1}
I_1 &\leq \varepsilon \sum_{\mathbf k\in\Lambda_n^{(1)}}
\sigma(n\mathbf s-\mathbf k).
\end{align}

For the $\Lambda_n^{(2)}$ terms, we have
\[ \left| f( \mathbf s_{\mathbf k,n}^{(\boldsymbol\alpha,\boldsymbol\beta)})
- f(\mathbf s) \right|
\leq 2\|f\|_\infty, \]
and therefore
\begin{align}\label{eq:I_2}
I_2 \leq 2\|f\|_\infty \sum_{\mathbf k\in\Lambda_n^{(2)}}
\sigma(n\mathbf s-\mathbf k).
\end{align}
Combining the estimates \eqref{eq: I_1} and \eqref{eq:I_2} gives
\[
\begin{aligned}
\left| F_n^{(\boldsymbol\alpha,\boldsymbol\beta)} (f;\mathbf s) -
f(\mathbf s) \right|
&\leq \frac{\varepsilon}{(\sigma_\eta(1))^d} \sum_{\mathbf k\in\Lambda_n}
\sigma(n\mathbf s-\mathbf k) \\
&\quad+ \frac{2\|f\|_\infty}{(\sigma_\eta(1))^d}
\sum_{\substack{\mathbf k\in\Lambda_n\\
\|n\mathbf s-\mathbf k\|_\infty>n\delta/2}}
\sigma(n\mathbf s-\mathbf k).
\end{aligned}
\]
Using \( \sum\limits_{\mathbf k\in\Lambda_n} \sigma(n\mathbf s-\mathbf k)
\leq \bigl(\mathcal M_0(\sigma_\eta)\bigr)^d \) and (iv) of Lemma~\ref{lemma-psi-properties}, we obtain
\[ \limsup_{n\to\infty} \sup_{\mathbf s\in K}
\left| F_n^{(\boldsymbol\alpha,\boldsymbol\beta)} (f;\mathbf s) - f(\mathbf s) \right| \leq \frac{ \bigl(\mathcal M_0(\sigma_\eta)\bigr)^d }
{(\sigma_\eta(1))^d}\varepsilon.
\]
Since \(\varepsilon>0\) is arbitrary, we conclude that
\[ \lim_{n\to\infty} \sup_{\mathbf s\in K} \left| F_n^{(\boldsymbol\alpha,\boldsymbol\beta)}
(f;\mathbf s) - f(\mathbf s) \right|=0.
\]
Hence, the convergence is uniform on \(K\).
\end{proof}

\begin{theorem}\label{thm:rate-convergence}
Assume that the activation function satisfies {\rm (P1)-(P3)} with $\rho>1$, and let $f\in C(K^*)$. Then there exists a constant $C>0$, independent of
$f$, $n$, and $\mathbf s\in K$, such that
\[ \left| F_n^{(\boldsymbol\alpha,\boldsymbol\beta)}
(f;\mathbf s) - f(\mathbf s) \right| \leq C\,
\omega\!\left(f,\frac1n\right).
\]
\end{theorem}

\begin{proof}
Let \(\mathbf s\in K\). By the definition of the operator and
Lemma~\ref{lem:denominator}, we obtain
\[ \begin{aligned}
\left| F_n^{(\boldsymbol\alpha,\boldsymbol\beta)}
(f;\mathbf s) - f(\mathbf s) \right|
&\leq \frac{1}{(\sigma_\eta(1))^d} \sum_{\mathbf k\in\Lambda_n}
\sigma(n\mathbf s-\mathbf k) \left| f( \mathbf s_{\mathbf k,n}^{(\boldsymbol\alpha,\boldsymbol\beta)}) - f(\mathbf s)
\right|.
\end{aligned}
\]
% By the definition of the modulus of continuity,
% \[ \left| f( \mathbf s_{\mathbf k,n}^{(\boldsymbol\alpha,\boldsymbol\beta)})
% - f(\mathbf s) \right| \leq \omega \left( f, \left\| \mathbf s_{\mathbf k,n}^{(\boldsymbol\alpha,\boldsymbol\beta)} - \mathbf s \right\|_\infty
% \right).
% \]
From Lemma~\ref{lem:stancu-node}, we have
\[
\left\|
\mathbf s_{\mathbf k,n}^{(\boldsymbol\alpha,\boldsymbol\beta)}
-
\mathbf s
\right\|_\infty
\leq
\left\|
\frac{\mathbf k}{n}-\mathbf s
\right\|_\infty
+
\frac{\widetilde  C}{n}.
\]
Using the standard inequality \( \omega(f,\lambda t) \leq (1+\lambda)\omega(f,t),
\ \lambda>0, \) we have
\[ 
\begin{aligned}
\omega \left( f, \left\| \mathbf s_{\mathbf k,n}^{(\boldsymbol\alpha,\boldsymbol\beta)}
- \mathbf s \right\|_\infty \right)
&\leq \omega \left( f, \left\| \frac{\mathbf k}{n}-\mathbf s
\right\|_\infty + \frac{\widetilde C}{n} \right)\\
&\leq (1+\widetilde  C) \omega \left( f, \left\| \frac{\mathbf k}{n}-\mathbf s
\right\|_\infty + \frac1n \right).
\end{aligned}
\]
Now, since \( \left\| \frac{\mathbf k}{n}-\mathbf s \right\|_\infty =
\frac{1}{n}\|n\mathbf s-\mathbf k\|_\infty, \) we obtain
\[
\begin{aligned}
\omega\!\left(
f,
\left\|
\mathbf s_{\mathbf k,n}^{(\boldsymbol\alpha,\boldsymbol\beta)}
-\mathbf s
\right\|_\infty
\right)
&\le
\omega\!\left(
f,
\frac{\|n\mathbf s-\mathbf k\|_\infty+\widetilde C}{n}
\right) \\
&\le
\left(
1+\|n\mathbf s-\mathbf k\|_\infty+\widetilde C
\right)
\omega\!\left(f,\frac{1}{n}\right).
\end{aligned}
\]
Therefore,
\begin{align}\label{eq: rate_1}
\left|
F_n^{(\boldsymbol\alpha,\boldsymbol\beta)} (f;\mathbf s)
- f(\mathbf s) \right|
&\leq \frac{\omega(f,1/n)} {(\sigma_\eta(1))^d}
\sum_{\mathbf k\in\Lambda_n}
\sigma(n\mathbf s-\mathbf k)
\left(1+\|n\mathbf s-\mathbf k\|_\infty+\widetilde C\right).
\end{align}

Since \( 1+\|n\mathbf s-\mathbf k\|_\infty+\widetilde C
\le (1+\widetilde C)\bigl(1+\|n\mathbf s-\mathbf k\|_\infty\bigr), \) and $1+\|n\mathbf s-\mathbf k\|_\infty
\le \prod\limits_{i=1}^{d}(1+|ns_i-k_i|),$ we obtain
\[
\begin{aligned}
\sum_{\mathbf k\in\Lambda_n} \sigma(n\mathbf s-\mathbf k)
\left(1+\|n\mathbf s-\mathbf k\|_\infty+\widetilde C\right)
&\le (1+\widetilde C) \sum_{\mathbf k\in\Lambda_n}
\sigma(n\mathbf s-\mathbf k) (1+\|n\mathbf s-\mathbf k\|_\infty)\\
&\le (1+\widetilde C) \prod_{i=1}^{d}\left( \sum_{k_i\in\mathbb Z} \sigma_\eta(ns_i-k_i)(1+|ns_i-k_i|)
\right).
\end{aligned}
\]
Using the discrete moments,
\[ \sum_{k_i\in\mathbb Z} \sigma_\eta(ns_i-k_i)
(1+|ns_i-k_i|) \le \mathcal M_0(\sigma_\eta) + \mathcal M_1(\sigma_\eta),
\]
and therefore
\[
\sum_{\mathbf k\in\Lambda_n} \sigma(n\mathbf s-\mathbf k)
(1+\|n\mathbf s-\mathbf k\|_\infty) \le \left( \mathcal M_0(\sigma_\eta) + \mathcal M_1(\sigma_\eta) \right)^d.
\]
Consequently, from~\eqref{eq: rate_1} we get 
\[ \left| F_n^{(\boldsymbol\alpha,\boldsymbol\beta)}
(f;\mathbf s) - f(\mathbf s) \right| \le C\,\omega(f,1/n),
\]
where
\( C= \frac{ (1+\widetilde C) \left( \mathcal M_0(\sigma_\eta) +
\mathcal M_1(\sigma_\eta) \right)^d } {(\sigma_\eta(1))^d}.
\)
\end{proof}

\begin{corollary}
Let \( f\in {\rm Lip}(\gamma,K^*), \ 0<\gamma\leq1 . \) Then there exists a constant \(C>0\) such that
\[ \left|
F_n^{(\boldsymbol\alpha,\boldsymbol\beta)} (f;\mathbf s) -
f(\mathbf s) \right|
\leq C n^{-\gamma}, \quad \mathbf s\in K .
\]
\end{corollary}

\begin{proof}
Since \( f\in {\rm Lip}(\gamma,K^*), \) there exists \(L>0\) such that
\[ \omega(f,\delta)\leq L\delta^\gamma.\]
Applying Theorem~\ref{thm:rate-convergence}, we obtain
\[ \left|
F_n^{(\boldsymbol\alpha,\boldsymbol\beta)} (f;\mathbf s) -
f(\mathbf s) \right| \leq
CLn^{-\gamma}.
\]
Renaming \(CL\) as \(C\) completes the proof.
\end{proof}

\begin{remark}
Theorem~\ref{thm:rate-convergence} shows that the Stancu perturbation
preserves the approximation order of the underlying NNOs. The parameters \(\boldsymbol\alpha\) and \(\boldsymbol\beta\)
affect only the sampling locations and contribute only through the
\(O(n^{-1})\) displacement of the nodes. For \( \boldsymbol\alpha=\boldsymbol\beta=\mathbf0,
\) the operators reduce to the existing multivariate NNOs.
\end{remark}

\begin{theorem}\label{thm:perturbation}
Let \( f\in C(K^*). \) Then 
\[ \left|
F_n^{(\boldsymbol\alpha,\boldsymbol\beta)} (f;\mathbf s) -
F_n(f;\mathbf s) \right| \leq (1+\widetilde  C) \left( \frac{\mathcal M_0(\sigma_\eta)}
{\sigma_\eta(1)} \right)^d\, \omega \left( f,\frac1n
\right), \quad \mathbf s\in K .
\]
Consequently,
\[ \sup_{\mathbf s\in K} \left| F_n^{(\boldsymbol\alpha,\boldsymbol\beta)}
(f;\mathbf s) - F_n(f;\mathbf s) \right| =
O\!\left( \omega \left( f,\frac1n
\right) \right). \]
\end{theorem}

\begin{proof}
Let \(\mathbf s\in K\). From the definitions of the perturbed and
unperturbed operators and Lemma~\ref{lem:denominator}, we have
\[
\begin{aligned}
\left| F_n^{(\boldsymbol\alpha,\boldsymbol\beta)}
(f;\mathbf s) - F_n(f;\mathbf s) \right| 
& = \left| \frac{ \displaystyle \sum_{\mathbf k\in\Lambda_n}
\sigma(n\mathbf s-\mathbf k) \left[ f \left( \mathbf s_{\mathbf k,n}^{(\boldsymbol\alpha,\boldsymbol\beta)} \right) - f \left( \frac{\mathbf k}{n} \right) \right] } {D_n(\mathbf s)}
\right|\\
&\leq \frac{1}{(\sigma_\eta(1))^d} \sum_{\mathbf k\in\Lambda_n}
\sigma(n\mathbf s-\mathbf k) \times
\left| f \left( \mathbf s_{\mathbf k,n}^{(\boldsymbol\alpha,\boldsymbol\beta)}
\right) - f \left( \frac{\mathbf k}{n}
\right) \right|.
\end{aligned}
\]
Since both points $\mathbf s_{\mathbf k,n}^{(\boldsymbol\alpha,\boldsymbol\beta)}$ and $ \frac{\mathbf k}{n}$ belong to \(K^*\), by the definition of the modulus of
continuity, 
\[ \left| f \left( \mathbf s_{\mathbf k,n}^{(\boldsymbol\alpha,\boldsymbol\beta)}
\right) - f \left( \frac{\mathbf k}{n}
\right) \right| \leq \omega \left( f, \left\| \mathbf s_{\mathbf k,n}^{(\boldsymbol\alpha,\boldsymbol\beta)}
- \frac{\mathbf k}{n} \right\|_\infty
\right).
\]
Now, by the node perturbation estimate~\ref{lem:stancu-node}, there exists \(\widetilde  C>0\) such that
\[ \left\| \mathbf s_{\mathbf k,n}^{(\boldsymbol\alpha,\boldsymbol\beta)}
- \frac{\mathbf k}{n} \right\|_\infty \leq
\frac{\widetilde  C}{n}.
\]
Therefore,
\[ \left| f \left(
\mathbf s_{\mathbf k,n}^{(\boldsymbol\alpha,\boldsymbol\beta)}
\right) - f \left( \frac{\mathbf k}{n} \right) \right|
\leq \omega \left( f,\frac{\widetilde  C}{n} \right).
\]
Substituting this estimate gives
\[
\begin{aligned}
\left| F_n^{(\boldsymbol\alpha,\boldsymbol\beta)} (f;\mathbf s)
- F_n(f;\mathbf s) \right| \leq \frac{ \omega(f,\widetilde  C/n) }
{(\sigma_\eta(1))^d } \sum_{\mathbf k\in\Lambda_n}
\sigma(n\mathbf s-\mathbf k).
\end{aligned}
\]
Using \( \sum\limits_{\mathbf k\in\Lambda_n} \sigma(n\mathbf s-\mathbf k) \leq
\bigl(\mathcal M_0(\sigma_\eta)\bigr)^d, \) we obtain
\[ \left|
F_n^{(\boldsymbol\alpha,\boldsymbol\beta)} (f;\mathbf s)
- F_n(f;\mathbf s) \right|
\leq \left( \frac{\mathcal M_0(\sigma_\eta)}
{\sigma_\eta(1)} \right)^d \omega(f,\widetilde C/n).
\]
Finally, applying the standard modulus inequality \( \omega(f,\lambda t) \leq
(1+\lambda)\omega(f,t), \ \lambda>0, \) with \( t=\frac1n, \ \lambda=\widetilde C, \) yields
\[
\omega(f,\widetilde C/n) \leq (1+\widetilde C)\omega(f,1/n).
\]
Hence,
\[
\left| F_n^{(\boldsymbol\alpha,\boldsymbol\beta)}
(f;\mathbf s) - F_n(f;\mathbf s) \right|
\leq (1+\widetilde  C) \left( \frac{\mathcal M_0(\sigma_\eta)}
{\sigma_\eta(1)} \right)^d\, \omega(f,1/n).
\]
\end{proof}

\begin{remark}
Theorem~\ref{thm:perturbation} shows that the Stancu-type modification
introduces only a controlled perturbation of the existing NNOs. The perturbation originates solely from the displacement of the
sampling nodes, which is of order \(O(n^{-1})\).

Therefore, the Stancu parameters \(\boldsymbol\alpha\) and \(\boldsymbol\beta\)
do not alter the asymptotic approximation order. They only influence the constant associated with the perturbation estimate.
\end{remark}

\section{Numerical Validation}\label{sec:numerical}

In this section, we present numerical experiments to validate the
approximation behaviour of the proposed Stancu-type NNOs. All computations are performed using MATLAB.

For completeness, we first describe the computational procedure for evaluating the Stancu-type NNOs.

\begin{algorithm}[H]
\caption{Stancu-Type NNOs}
\label{alg:stancu_nno}
\renewcommand{\baselinestretch}{0.85}\selectfont
\begin{algorithmic}[1]

\Require $f:K^*\to\mathbb R$, $\boldsymbol\alpha,\boldsymbol\beta$ with
$0\le \boldsymbol\alpha\le\boldsymbol\beta$, $n\in\mathbb N$,
$\sigma_\eta$, $\mathbf s\in K$

\Ensure $F_n^{(\boldsymbol\alpha,\boldsymbol\beta)}(f;\mathbf s)$

\State $\Lambda_n=\{\mathbf k\in\mathbb Z^d:
\lceil na_i\rceil\le k_i\le\lfloor nb_i\rfloor,\ i=1,\ldots,d\}$

\State $\mathrm{Num}\gets0,\quad \mathrm{Den}\gets0$

\For{$\mathbf k\in\Lambda_n$}

\State $w_{\mathbf k}(\mathbf s)\gets\sigma(n\mathbf s-\mathbf k)$

\State $\mathbf s_{\mathbf k,n}^{(\boldsymbol\alpha,\boldsymbol\beta)}
\gets
\left(
\frac{k_1+\alpha_1}{n+\beta_1},\ldots,
\frac{k_d+\alpha_d}{n+\beta_d}
\right)$

\State $\mathrm{Num}\gets \mathrm{Num}
+w_{\mathbf k}(\mathbf s)
f\!\left(\mathbf s_{\mathbf k,n}^{(\boldsymbol\alpha,\boldsymbol\beta)}\right)$

\State $\mathrm{Den}\gets \mathrm{Den}+w_{\mathbf k}(\mathbf s)$

\EndFor

\State $F_n^{(\boldsymbol\alpha,\boldsymbol\beta)}(f;\mathbf s)
\gets \mathrm{Num}/\mathrm{Den}$

\State \Return $F_n^{(\boldsymbol\alpha,\boldsymbol\beta)}(f;\mathbf s)$

\end{algorithmic}
\end{algorithm}

In all experiments, we consider the compact interval $K=[0,1]$. The activation mechanism is generated through the sigmoidal function \[ \eta(s)=\frac{1}{1+e^{-s}}, \]
from which we construct the associated neural activation function
\[
\sigma_\eta(s)=\frac{1}{2}\left[\eta(s+1)-\eta(s-1)\right].
\]
To evaluate the approximation performance of the proposed Stancu-type NNOs, we consider two representative test functions exhibiting different characteristics. 

For the numerical experiments, we consider the following smooth test functions:
\[
\begin{aligned}
f(s) ={}& e^{-100(s-0.2)^2} +\sin(2\pi s)e^{-80(s-0.5)^2} \\
&+\sin(4\pi s)e^{-120(s-0.8)^2} +\sin(6\pi s)e^{-8(2s-1)^2}, \quad s\in[0,1].
\end{aligned}
\]
For the two-dimensional case, we consider 
\[
\begin{aligned}
g(s,t)={}& 0.75e^{-\frac{(9s-2)^2+(9t-2)^2}{4}}
+0.75e^{-\frac{(9s+1)^2}{49} -\frac{9t+1}{10}}+0.5e^{ -\frac{(9s-7)^2+(9t-3)^2}{4}} \\
&\qquad -0.2e^{-(9s-4)^2-(9t-7)^2}+0.15\cos\!\left( 12\sqrt{(s-0.5)^2+(t-0.5)^2}
\right), \ (s,t)\in[0,1]^2.
\end{aligned}
\]

\subsection{Influence of the Stancu Parameters}

We first investigate the influence of the parameters $\alpha$ and $\beta$ on the approximation behaviour. Figures~\ref{fig:approximation} and~\ref{fig:approximation_2D}  shows the approximation of the function $f(s)$ and $g(s,t)$ using the Stancu-type NNOs for $n=50$ and three different parameter choices:
\[ (\boldsymbol \alpha,\boldsymbol \beta)=(\boldsymbol 0,\boldsymbol 0),\quad (\boldsymbol {0.5} , \boldsymbol {0.5}),\quad (\boldsymbol 1,\boldsymbol 2). \]

The results show that all parameter choices provide accurate approximations of the function. However, the shape of the approximation varies slightly with the selected parameters. This behaviour is due to the perturbation of the sampling nodes
\[ \mathbf s_{\mathbf k,n}^{(\boldsymbol{\alpha},\boldsymbol{\beta})} =
\left( \frac{k_1+\alpha_1}{n+\beta_1}, \ldots, \frac{k_d+\alpha_d}{n+\beta_d}
\right). \]
which introduces additional flexibility compared with the existing NNOs.

\begin{figure}[H]
\centering
\includegraphics[width=0.8\textwidth]{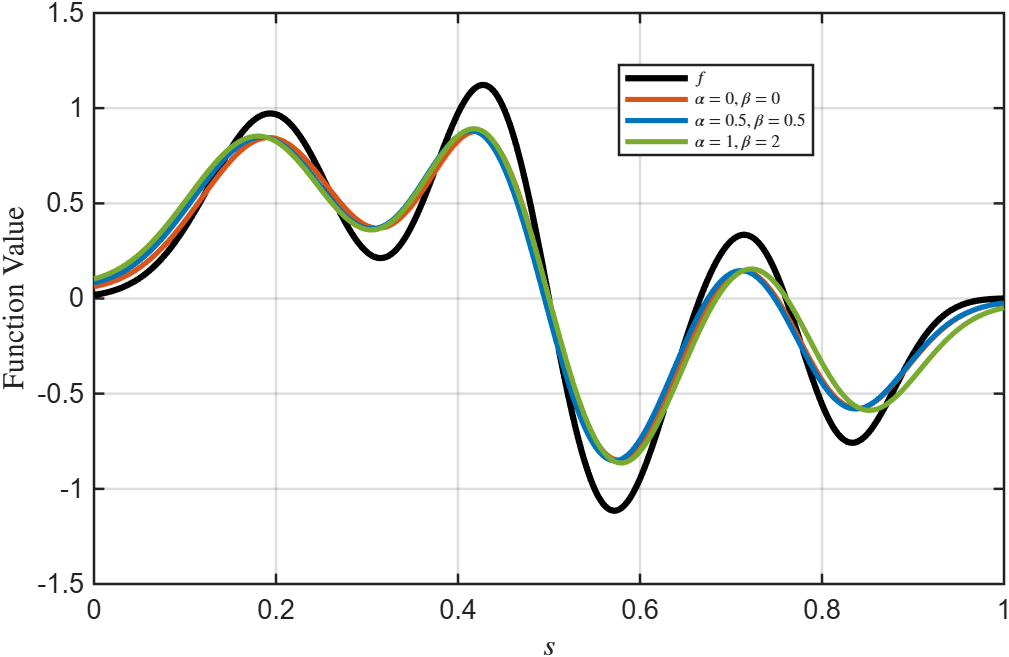}
\caption{Approximation of $f(s)$ using the Stancu-type NNOs for different values of $(\alpha,\beta)$ with $n=50$.}
\label{fig:approximation}
\end{figure}

\begin{figure}[H]
\centering
\includegraphics[width=1\textwidth]{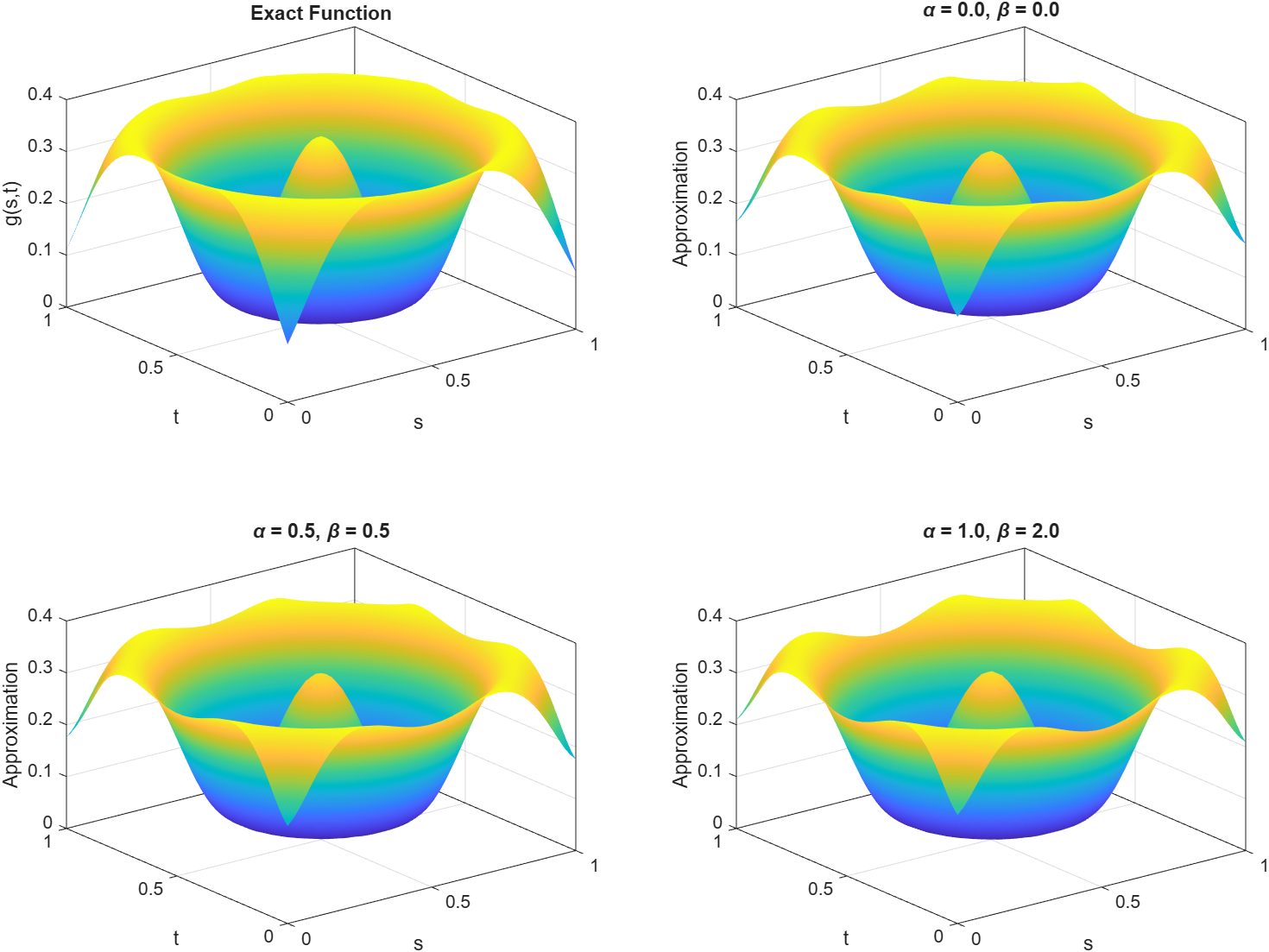}
\caption{Approximation of $g(s,t)$ using the 2D-Stancu-type NNOs for different values of $(\alpha,\beta)$ with $n=50$.}
\label{fig:approximation_2D}
\end{figure}

\subsection{Errors}
Next, we analyze the convergence of the Stancu-type NNOs. For fixed parameters $(\alpha,\beta)=(0.5,0.5)$, we compute the maximum approximation error
\[ E_n=\max_{s\in[0,1]}\left|F_n^{(\alpha,\beta)}(f;s)-f(s)\right| \]
for $n=10,20,\ldots,1000$.

Figure~\ref{fig:convergence} illustrates the decay of the error as $n$ increases. The results clearly show that the approximation error decreases steadily, confirming the theoretical convergence results established in the previous section. 
\begin{figure}[H]
\centering
\includegraphics[width=1\textwidth]{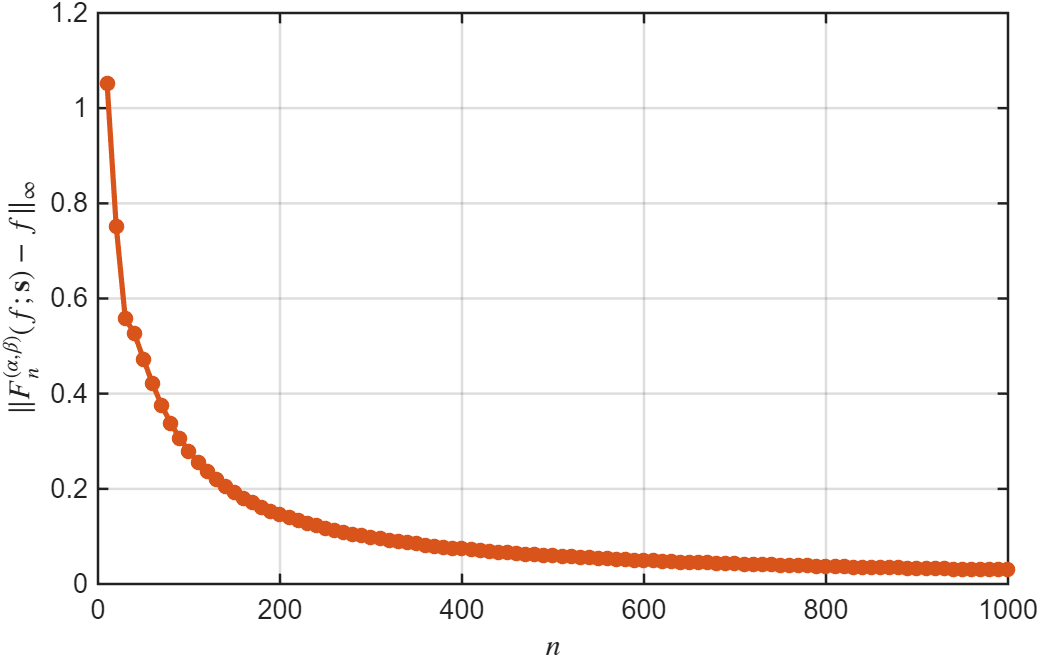}
\caption{Decay of the maximum approximation error for the Stancu-type NNOs as $n$ increases.}
\label{fig:convergence}
\end{figure}

Figure~\ref{fig:error2D} and Table~\ref{tab:stancu_errors} illustrate the variation of the absolute maximum approximation error for the 2D-Stancu-type NNOs as the number of nodes increases.

\begin{figure}[H]
\centering
\includegraphics[width=1\textwidth]{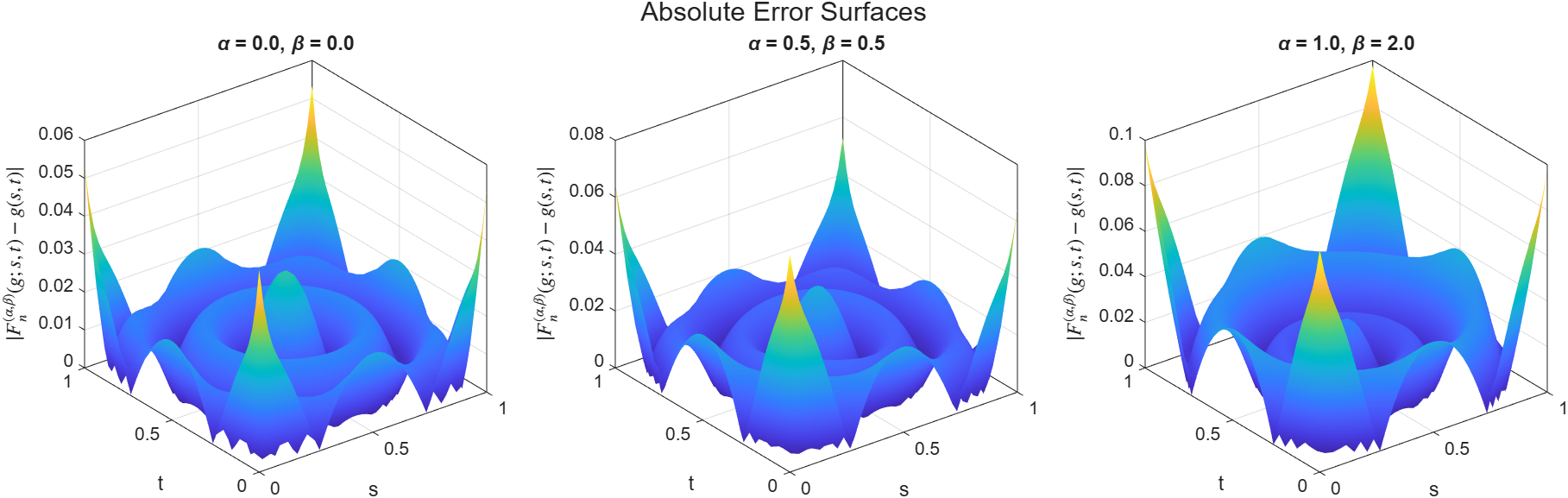}
\caption{Absolute maximum error for the 2D-Stancu-type NNOs.}
\label{fig:error2D}
\end{figure}

\begin{table}[th]
\centering
\begin{tabular}{|c|c|c|}
\hline
$(\alpha_1,\alpha_2)$ & $(\beta_1,\beta_2)$ & Max Error \\
\hline
(0,0) & (0,0) & $4.757178\times10^{-2}$ \\
(0.50,0.50) & (0.50,0.50) & $6.827981\times10^{-2}$ \\
(1.00,1.00) & (2.00,2.00) & $8.887727\times10^{-2}$ \\
\hline
\end{tabular}
\caption{Maximum errors for different choices of the Stancu parameters $(\boldsymbol \alpha,\boldsymbol \beta)$.}
\label{tab:stancu_errors}
\end{table}

% \subsection{Shift of Sampling Nodes}

% Finally, we illustrate the shift of the sampling nodes introduced by the Stancu parameters. In the existing NNOs, the sampling nodes are given by
% \[
% s_{k,n}=\frac{k}{n},
% \]
% which are uniformly distributed over the interval. In contrast, the Stancu-type NNOs employ the modified nodes
% \[
% s_{k,n}^{(\alpha,\beta)}=\frac{k+\alpha}{n+\beta}.
% \]

% Figure~\ref{fig:nodeshift} compares the existing nodes with the Stancu nodes for $(\alpha,\beta)=(0.5,0.5)$. The figure clearly shows that the Stancu modification produces a slight shift of the sampling locations, providing additional flexibility in the approximation process.

% \begin{figure}[H]
% \centering
% \includegraphics[width=0.8\textwidth]{Figure_3.png}
% \caption{Comparison between original sampling nodes $k/n$ and Stancu nodes $(k+\alpha)/(n+\beta)$.}
% \label{fig:nodeshift}
% \end{figure}

\subsection{Discussion}

The numerical experiments confirm the theoretical results obtained in the previous sections. In particular, the experiments demonstrate that the proposed Stancu-type NNOs provide accurate approximations even for functions exhibiting both oscillatory behavior and non-smooth features. Moreover, the parameters $\alpha$ and $\beta$ introduce a controllable perturbation of the sampling nodes, allowing additional flexibility in the approximation scheme.

\begin{remark}
It is well established that the logistic function constitutes one of the most widely used activation functions in the study of NNOs; see, for instance, \cite{qian2022rates,lin2026construction,saini2026approximation}. Besides the logistic function, several other activation functions have also been investigated, including the hyperbolic tangent sigmoid and ramp functions, as discussed in \cite{karateke2025some}. Notably, the latter class of functions is closely connected to the celebrated ReLU (Rectified Linear Unit) activation function; see, e.g., \cite{devore2021neural}.
\end{remark}

\subsection{Application to Approximation from Perturbed Samples}
In many practical situations, the sampling points of a signal or dataset are not perfectly uniformly distributed. Instead, the observed samples may be slightly shifted due to measurement errors, sensor misalignment, or irregular sampling grids. The Stancu-type NNOs naturally accommodate such situations through the modified sampling nodes
\[ \mathbf s_{\mathbf k,n}^{(\boldsymbol{\alpha},\boldsymbol{\beta})}
= \left( \frac{k_1+\alpha_1}{n+\beta_1}, \ldots, \frac{k_d+\alpha_d}{n+\beta_d}
\right).
\]
These nodes represent a controlled perturbation of the original uniform nodes $k/n$. Consequently, the proposed operators can be used to reconstruct functions from perturbed or non-uniform sampling locations. This flexibility makes the Stancu modification particularly suitable for applications involving real-world sampled data.

\subsection{Illustrative Example: ECG Signal Reconstruction}
In this subsection, we demonstrate the applicability of the proposed Stancu-type NNOs in the problem of signal reconstruction. In many practical situations, biomedical signals such as electrocardiogram (ECG) recordings are affected by measurement noise due to sensor inaccuracies, electronic interference, or motion artifacts. Therefore, it is important to reconstruct a smooth approximation of the underlying signal from perturbed observations.

Let $f:[0,1]\to\mathbb{R}$ denote the true signal. In practice, the available samples are contaminated with noise and can be written as
\[ y_k = f + \varepsilon_k, \quad k=0,1,\dots,n-1, \]
where $\varepsilon_k$ represents additive noise.

To reconstruct the signal from the perturbed data set $\{ \frac{k+\alpha}{n+\beta}\}$, we employ the Stancu-type NNOs defined in the previous section, namely
\[
F_{n}^{(\alpha,\beta)}(y_k;s) = \frac{ \sum\limits_{\mathbf{k}\in\Lambda_n}
\sigma(ns-k)\, y_k\!\left(\frac{k+\alpha}{n+\beta}\right) }{
\sum\limits_{\mathbf{k}\in\Lambda_n} \sigma(ns-k) }, \quad s\in[0,1].
\]
The parameters $\alpha$ and $\beta$ introduce a controlled perturbation of the sampling nodes and provide additional flexibility in the reconstruction process.

\medskip

For the numerical experiment, an ECG signal is generated
using Gaussian functions that model the characteristic components
of an ECG waveform, namely the $P$, $Q$, $R$, $S$, and $T$ waves.
Specifically, the signal is given by
\begin{equation}\nonumber
\begin{aligned}
f(s) ={}& 1.2\,e^{-((s-0.25)/0.03)^2}
        -2.5\,e^{-((s-0.30)/0.01)^2} \\
       & +4.0\,e^{-((s-0.32)/0.008)^2}
        -1.8\,e^{-((s-0.35)/0.015)^2} \\
       & +1.5\,e^{-((s-0.60)/0.05)^2}.
\end{aligned}
\end{equation}
Gaussian noise is added to the signal in order to simulate
realistic measurement conditions. The Stancu-type NNOs is then applied to the noisy data in order to obtain
a reconstruction of the signal.

Figure~\ref{fig:ecg_denoising} shows the true ECG signal, the perturbed observations, and the reconstructed signal obtained by the proposed operator. It can be observed that the Stancu-type NNOs effectively suppresses noise
while preserving the main morphological features of the ECG
waveform, including the $P$ wave, the $QRS$ complex, and the
$T$ wave.

\begin{figure}[H]
\centering
\includegraphics[width=0.85\textwidth]{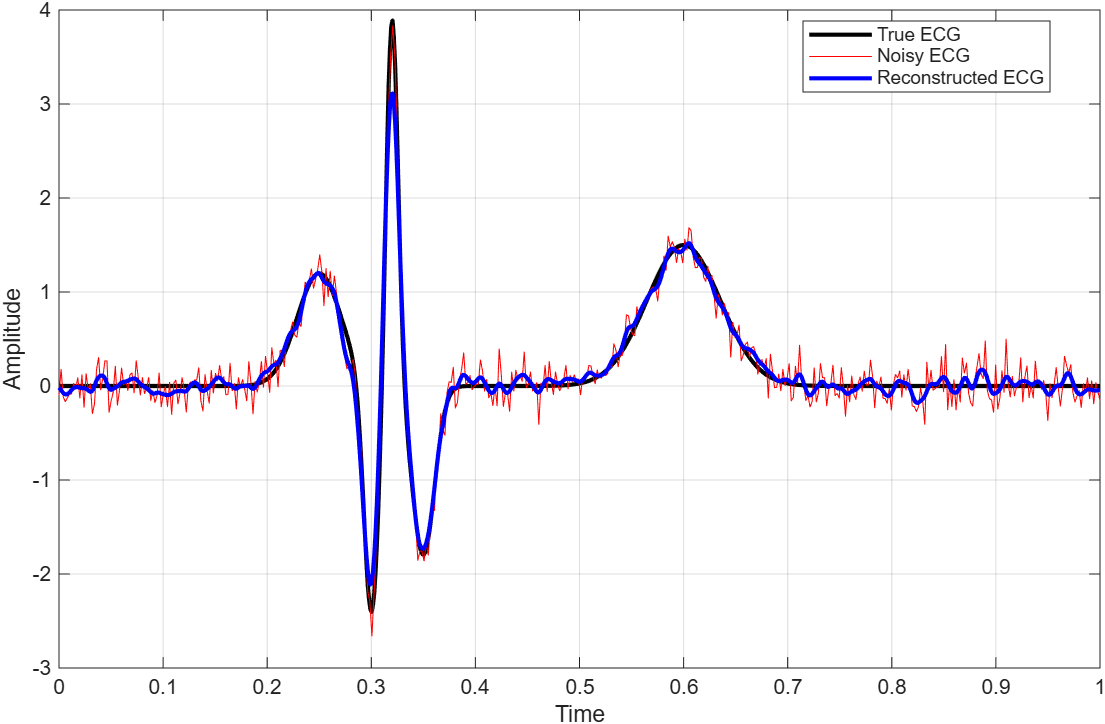}
\caption{ECG signal using the Stancu-type NNOs.
The black curve represents the true ECG signal, and the red curve corresponds
to the perturbed signal, and the blue curve shows the reconstructed
signal obtained by the proposed operator for $n=500$.}
\label{fig:ecg_denoising}
\end{figure}

The results illustrate that the proposed Stancu-type NNOs provide an effective tool for signal reconstruction, thanks
to the localized averaging property of the activation function
and the flexibility introduced by the Stancu parameters
$\alpha$ and $\beta$.

To quantitatively evaluate the performance of the proposed
Stancu-type NNOs in the reconstruction task, we compute the root mean square error (RMSE) between the reconstructed signal
and the true ECG signal. The RMSE is defined as

\begin{equation}\nonumber
\mathrm{RMSE} = \sqrt{ \frac{1}{n} \sum_{i=1}^{n} \left( F_{n}^{(\alpha,\beta)}(y;s_i) - y(s_i) \right)^2
},
\end{equation}
where $y(s_i)$ denotes the true ECG signal evaluated at the reconstruction points $s_i$, and $F_{n}^{(\alpha,\beta)}(y;s_i)$ represents the signal reconstructed
by the Stancu-type NNOs.

Table~\ref{tab:rmse_n} reports the RMSE values obtained for different sample sizes $n$ with fixed parameters $\alpha=0.5$, $\beta=1$, and noise level $0.15$. It can be observed that the reconstruction error generally decreases as the number of samples increases, confirming the
convergence behaviour of the proposed Stancu-type NNOs.

\begin{table}[th]
\centering
\begin{tabular}{|c|c|}
\hline
$n$ & RMSE \\
\hline
100  & 0.37633 \\
200  & 0.25218 \\
400  & 0.13451 \\
600  & 0.07969 \\
800  & 0.06650 \\
1000 & 0.062801 \\
\hline
\end{tabular}
\caption{RMSE values for different sample sizes $n$ obtained using the proposed Stancu-type NNOs with parameters $\alpha=0.5$, $\beta=1$, and noise level $0.15$.}
\label{tab:rmse_n}
\end{table}
\begin{remark}
    
The motivation for introducing the Stancu-type modification is twofold. First, in many practical situations, the available samples are not located exactly at the uniform grid points \(k/n\), but rather at perturbed locations arising from measurement inaccuracies, calibration errors,
or asynchronous data acquisition. The parameters \(\alpha\) and \(\beta\) provide a simple deterministic mechanism for modelling such perturbations.

Second, from a theoretical perspective, Stancu-type operators enable the investigation of the stability and robustness of NNOs under systematic perturbations of the sampling nodes. In particular, the perturbation estimates established in this paper show that the approximation order of the existing NNOs is preserved under Stancu-type modifications.
\end{remark}

\section{Conclusion}

In this paper, we introduced a Stancu-type generalization of multivariate NNOs by incorporating parameters that perturb the sampling nodes. The proposed operators provide additional flexibility in the approximation process while preserving the essential structure of NNOs.

We established the well-definedness and boundedness of the operators and proved their uniform convergence on compact domains. Furthermore, quantitative error estimates were derived in terms of the modulus of continuity, leading to convergence rate results. A perturbation theorem is also derived, demonstrating the deviation between the proposed Stancu-type operators and the corresponding existing NNOs.

Numerical experiments were conducted to validate the theoretical findings and illustrate the influence of the Stancu parameters on the approximation behaviour and the sampling nodes. The results confirm the effectiveness and flexibility of the proposed operators. Furthermore, the numerical experiment on ECG signal reconstruction confirms the practical effectiveness of the proposed Stancu-type NNOs, showing stable reconstruction performance even in the presence of noise. Future work may include extensions to Kantorovich-type variants, stochastic NNOs.

\bigskip
\noindent\textbf{Declaration of competing interest}

The author declares that he has no competing financial interests or personal relationships that could influence the reported work in this paper.

\noindent\textbf{Data availability}

No datasets were generated or analyzed during the current study.

\noindent\textbf{Funding}

The author declares that no funds, grants, or other financial support were received for the research, authorship, or publication of this work.

\bibliographystyle{elsarticle-num}
\bibliography{Ref}

@article{cardaliaguet1992approximation,
  title={Approximation of a function and its derivative with a neural network},
  author={Cardaliaguet, Pierre and Euvrard, Guillaume},
  journal={Neural Netw.},
  volume={5},
  number={2},
  pages={207--220},
  year={1992},
  publisher={Elsevier}
}

@article{anastassiou1997rate,
  title={Rate of convergence of some neural network operators to the unit-univariate case},
  author={Anastassiou, George A},
  journal={J. Math. Anal. Appl.},
  volume={212},
  number={1},
  pages={237--262},
  year={1997},
  publisher={Elsevier}
}

@article{anastassiou2000rate,
  title={Rate of convergence of some multivariate neural network operators to the unit},
  author={Anastassiou, George A},
  journal={Comput. Math. Appl.},
  volume={40},
  number={1},
  pages={1--19},
  year={2000},
  publisher={Elsevier}
}

@article{anastassiou2011multivariatesigmoidal,
  title={Multivariate sigmoidal neural network approximation},
  author={Anastassiou, George A},
  journal={Neural Netw.},
  volume={24},
  number={4},
  pages={378--386},
  year={2011},
  publisher={Elsevier}
}

@article{anastassiou2013multi-vari-rate,
  title={Rate of convergence of some multivariate neural network operators to the unit, revisited.},
  author={Anastassiou, George A},
  journal={J. Comput. Anal. Appl.},
  volume={15},
  number={1},
  pages={1300--1309},
  year={2013}
}

@article{anastassiou2013frac-normal,
  title={Fractional approximation by normalized bell and squashing type neural network operators},
  author={Anastassiou, George A},
  journal={New Math. Nat. Comput.	},
  volume={9},
  number={01},
  pages={43--63},
  year={2013},
  publisher={World Scientific}
}

@article{costarelli2013multivariate,
  title={Multivariate neural network operators with sigmoidal activation functions},
  author={Costarelli, Danilo and Spigler, Renato},
  journal={Neural Netw.},
  volume={48},
  pages={72--77},
  year={2013},
  publisher={Elsevier}
}

@article{costarelli2014convergence,
  title={Convergence of a family of neural network operators of the Kantorovich type},
  author={Costarelli, Danilo and Spigler, Renato},
  journal={J. Approx. Theory},
  volume={185},
  pages={80--90},
  year={2014},
  publisher={Elsevier}
}

@article{costarelli2013approxi.Single-layer-for-C-class,
  title={Approximation results for neural network operators activated by sigmoidal functions},
  author={Costarelli, Danilo and Spigler, Renato},
  journal={Neural Netw.},
  volume={44},
  pages={101--106},
  year={2013},
  publisher={Elsevier}
}

@article{costarelli2022density,
  title={Density results by deep neural network operators with integer weights},
  author={Costarelli, Danilo},
  journal={Math. Model. Anal.},
  volume={27},
  number={4},
  pages={547--560},
  year={2022}
}

@article{costarelli2022quantitative,
  title={Quantitative estimates for neural network operators implied by the asymptotic behaviour of the sigmoidal activation functions},
  author={Coroianu, Lucian and Costarelli, Danilo and Kadak, U{\u{g}}ur},
  journal={Mediterr. J. Math.},
  volume={19},
  number={5},
  pages={211},
  year={2022},
  publisher={Springer}
}

@article{costarelli2024Durrmeyer,
  title={The approximation capabilities of Durrmeyer-type neural network operators},
  author={Coroianu, Lucian and Costarelli, Danilo and Natale, Mariarosaria and Panti{\c{s}}, Alexandra},
  journal={J. Appl. Math. Comput.},
  pages={1--19},
  year={2024},
  publisher={Springer}
}

@article{nasiruzzaman2025approximation,
  title={Approximation by Stancu-type $\alpha$-Bernstein-Schurer-Kantorovich operators},
  author={Nasiruzzaman, Md},
  journal={J. Inequal. Appl.},
  volume={2025},
  number={1},
  pages={48},
  year={2025},
  publisher={Springer}
}

@article{cai2024approximation,
  title={Approximation by a new Stancu variant of generalized ($\lambda$, $\mu$)-Bernstein operators},
  author={Cai, Qing-Bo and Aslan, Re{\c{s}}at and {\"O}zger, Faruk and Srivastava, Hari Mohan},
  journal={Alexandria Eng. J.},
  volume={107},
  pages={205--214},
  year={2024},
  publisher={Elsevier}
}

@article{torun2024some,
  title={Some Approximation Properties of the (p, q)-Stancu-Schurer-Bleimann-Butzer-Hahn Operators},
  author={Torun, G{\"u}lten},
  journal={J. Math.},
  volume={2024},
  number={1},
  pages={9083766},
  year={2024},
  publisher={Wiley Online Library}
}

@article{lian2023approximation,
  title={Approximation properties of Stancu type Sz{\'a}sz-Mirakjan operators},
  author={Lian, Bo-Yong and Cai, Qing-Bo},
  journal={AIMS Math},
  volume={8},
  number={9},
  pages={21769--21780},
  year={2023}
}

@article{berisha2025riemann,
  title={Riemann--Liouville fractional exponential sampling type neural network Kantorovich operators},
  author={Berisha, Artan and Baxhaku, Fesal and Agrawal, Purshottam Narain and Baxhaku, Behar},
  journal={Eng. Appl. Artif. Intell.},
  volume={151},
  pages={110477},
  year={2025},
  publisher={Elsevier}
}

@article{saini2026approximation,
  title={Approximation capabilities of certain nonlinear neural network operators},
  author={Saini, Sachin and Singh, Uaday},
  journal={J. Comput. Appl. Math.},
  pages={117452},
  year={2026},
  publisher={Elsevier}
}

@article{karateke2025some,
  title={Some mathematical properties of flexible hyperbolic tangent activation function with application to deep neural networks},
  author={Karateke, Seda},
  journal={Revista de la Real Academia de Ciencias Exactas, F{\'\i}sicas y Naturales. Serie A. Matem{\'a}ticas},
  volume={119},
  number={2},
  pages={42},
  year={2025},
  publisher={Springer}
}

@article{devore2021neural,
  title={Neural network approximation},
  author={DeVore, Ronald and Hanin, Boris and Petrova, Guergana},
  journal={Acta Numer.},
  volume={30},
  pages={327--444},
  year={2021},
  publisher={Cambridge University Press}
}

@article{lin2026construction,
  title={Construction and approximation rate of neural network operators for Broad Learning System},
  author={Lin, Huijie and Yu, Dansheng and Cao, Feilong},
  journal={Chaos, Solitons \& Fractals},
  volume={207},
  pages={118048},
  year={2026},
  publisher={Elsevier}
}

@article{qian2022rates,
  title={Rates of approximation by neural network interpolation operators},
  author={Qian, Yunyou and Yu, Dansheng},
  journal={Appl. Math. Comput.},
  volume={418},
  pages={126781},
  year={2022},
  publisher={Elsevier}
}

@article{yu2025construction,
  title={Construction and approximation properties of exact neural network interpolation operators activated by entire functions},
  author={Yu, Dansheng},
  journal={J. Approx. Theory},
  pages={106215},
  year={2025},
  publisher={Elsevier}
}

@article{costarelli2023approximation,
  title={Approximation error for neural network operators by an averaged modulus of smoothness},
  author={Costarelli, Danilo},
  journal={J. Approx. Theory},
  volume={294},
  pages={105944},
  year={2023},
  publisher={Elsevier}
}

@article{bhat2025boundary,
  title={Boundary interpolation on triangles via neural network operators},
  author={Bhat, Aaqib Ayoub and Khan, Asif},
  journal={Math. Comput. Simul.},
  year={2025},
  publisher={Elsevier}
}
\end{document}